\documentclass[10pt, leqno]{amsart}

\usepackage{amsfonts,amssymb}
\allowdisplaybreaks[4]

\usepackage{amsmath}
\usepackage{amsthm}
\usepackage{amsrefs}
\usepackage{qsymbols}
\usepackage{latexsym}
\usepackage{chngcntr}
\usepackage[noadjust]{cite}
\usepackage{paralist}
\usepackage{fixltx2e}
\usepackage{esint}

\newtheorem{theorem}{Theorem}[section]

\newtheorem{lemma}[theorem]{Lemma}
\newtheorem{proposition}[theorem]{Proposition}

\theoremstyle{definition}
\newtheorem{definition}[theorem]{Definition}

\newtheorem{innercustomgeneric}{\customgenericname}
\providecommand{\customgenericname}{}
\newcommand{\newcustomtheorem}[2]{%
  \newenvironment{#1}[1]
  {%
   \renewcommand\customgenericname{#2}%
   \renewcommand\theinnercustomgeneric{##1}%
   \innercustomgeneric
  }
  {\endinnercustomgeneric}
}

\newcustomtheorem{assumption}{Assumption}

\newcommand{\IR}{\mathbb{R}}
\newcommand{\IC}{\mathbb{C}}
\newcommand{\IN}{\mathbb{N}}

\newcommand{\IS}{\mathbb{S}}

\newcommand{\cE}{\mathcal{E}}
\newcommand{\cF}{\mathcal{F}}

\newcommand{\cR}{\mathcal{R}}
\newcommand{\cA}{\mathcal{A}}
\newcommand{\cB}{\mathcal{B}}
\newcommand{\cC}{\mathcal{C}}

\newcommand{\cT}{\mathcal{T}}
\newcommand{\HT}{\mathcal{HT}}

\newcommand{\abs}[1]{\left|#1\right|}

\newcommand{\g}{\Gamma}

\providecommand{\abs}[1]{|#1|}
\newcommand{\Complex}{ \mathbb{C} }

\newcommand{\Fourier}{ \mathcal{F}}
\newcommand{\FourierInverse}{ \mathcal{F}^{ - 1 } }

\newcommand{\Laplace}{\mathcal{L}}
\newcommand{\LaplaceInverse}{\mathcal{L}^{-1}}

\newcommand{\LongSet}[2]{\left \{ #1 : #2 \right \}}
\newcommand{\Prime}[1]{ {#1}^{\prime} }
\newcommand{\Real}{ \mathbb{R} }

\newcommand{\DomainOfLambda}{\Sigma_\theta}

\newcommand{\eps}{\varepsilon}

\renewcommand\Re{\operatorname{Re}}

\newcommand{\cl}[1]{\overline{#1}}

\numberwithin{equation}{section}

\title{Maximal $L_p$-$L_q$ Regularity for the Quasi-Steady Elliptic Problems}

\author{Ken Furukawa}
\author{Naoto Kajiwara}
\address{Graduate School of Mathematical Sciences, The University of Tokyo, 3-8-1 Komaba, Meguro, Tokyo, 153-8914, Japan}
\address{Department of Mathematics, Faculty of Science and Technology, Tokyo University of Science, Noda, Chiba, 278-8510, Japan}

\email{kenf@ms.u-tokyo.ac.jp}
\email{kajiwara\_naoto@ma.noda.tus.ac.jp}

\begin{document}

\begin{abstract}
In this paper we consider maximal regularity for the vector-valued quasi-steady linear elliptic problems. 
The equations are the elliptic equation in the domain and the evolution equations on its boundary. 
We prove the maximal $L_p$-$L_q$ regularity for these problems and give examples that our results are applicable. 
The Lopatinskii--Shapiro and the {\it asymptotic} Lopatinskii--Shapiro conditions are important to get boundedness of solution operators.
\end{abstract}

\maketitle

\section{Introduction}\label{Sec: Intro}
We consider the vector-valued quasi-steady problems of the following 
	\begin{align}\label{eq_qs}
		\left\{
			\begin{array}{rll}
				\eta u
				+ \cA(t,x,D) u 
				&= f(t,x) \quad 
				&(t\in J,~ x\in G), \\
				\partial_t \rho  
				+\cB_0(t,x,D) u + \cC_0(t,x,D_\g) \rho 
				&= g_0(t,x) \quad 
				&(t \in  J,~ x\in \g), \\
				\cB_j(t,x,D) u + \cC_j(t,x,D_\g) \rho 
				&= g_j(t,x) \quad 
				&(t\in J,~ x\in \g,~j = 1, \cdots, m), \\
				\rho(0,x) 
				&= \rho_0(x) \quad 
				&(x\in \g),
			\end{array}
		\right.
	\end{align}
where $\eta > 0$, $J \subset [0,T]$ is a finite interval and $G \subset \IR^n$ $(n \geq 2)$ is a bounded or an exterior domain with the boundary $\g$. 
The functions $f, \{g_j\}_{j=0}^m, \rho_0$ are given data and the functions $u$ and $\rho$ are unknown functions. 
$(\cA, \cB_j, \cC_j)$ are differential operators with order $(2m, m_j, k_j)$, respectively. 
The aim of this paper is to obtain maximal $L_p$-$L_q$ regularity of these equations. 
More precisely we characterize the data space $X \times \prod_{j =0}^m Y_j \times \pi Z_\rho$ and the solution space $Z_u\times Z_\rho$ such that these spaces are isomorphism. 

This quasi-steady problems are considered as the linearized equations for the various non-linear equations, e.g. free boundary problems. 
One of the successful methods to solve the free boundary problems is the transformation from time-varying domain to fixed domain. 
After we use this transformation, the equation has an unknown function called a height function on the boundary, and the equation on the boundary has time derivative of order one. 
If the original equation has a time derivative in an interior domain, the transformed equation also has a time derivative in a domain. 
On the other hand, if there is no time derivative in the original equation in the domain, the transformed equation does not have a time derivative. 
Usually, the derived equation is also non-linear, but the linearized equation corresponds to the relaxation type or the quasi-steady type. 
The first one corresponds to the first derivative in the interior equation and it has already considered in the paper \cite{DenkHieberPruss2003}. 
As far as we know, the second one has not considered yet. 
Therefore we consider these problems in this paper. 

The paper is organized as follows. 
In Section \ref{Sec: Main}, basic function spaces and assumptions for $\mathcal{A}$, $\mathcal{B}_j$ and $\mathcal{C}_j$ including smoothness are introduced.
Then our main result is stated. 
In Section \ref{Sec: Pre}, basic notions of operator theory, e.g. operator-valued multiplier theorems and $H^\infty$-calculus, are introduced for the reader's convenience.
In Section \ref{Sec: solvability}, we first consider (\ref{eq_qs}) under $G = \Real^n_+$ with the differential operators having no lower order terms and constant coefficients.
The problem is first reduced into the case of $f = 0$ and $\rho_0 = 0$.
Then the partial Laplace-Fourier transform is applied to get the solution formula of Fourier multiplier type.
In this step, the Lopatinskii--Shapiro condition (LS) is frequently used.
Operator-valued Fourier multiplier theorem due to Weis \cite{Weis2001} and the operator-valued $H^\infty$-functional calculus due to Kalton--Weis \cite{KaltonWeis2001} are applied to the solution operator to obtain its maximal regularity of the solutions.
Here, the asymptotic Lopatinskii--Shapiro (ALS) conditions are also needed.
By perturbation and localization procedure, our maximal regularity result for the full problem of (\ref{eq_qs}) is proved.

\section{Main results}\label{Sec: Main}
Let us introduce notation to give our main results and state our theorem.
Let $\mathbb{N}$ be a set of positive integer and $\mathbb{N}_0 = \mathbb{N} \cup \{ 0 \}$. 
Differential operators in (\ref{eq_qs}) are given by
	\begin{align*}
		\cA(t,x,D) 
		:= \sum_{|\alpha|\le 2m} a_\alpha(t,x) D^\alpha, \\
		\cB_j(t,x,D) 
		:= \sum_{|\beta|\le m_j} b_{j\beta}(t,x) D^\beta, \\
		\cC_j(t,x,D_\g) := \sum_{|\gamma|\le k_j} c_{j\gamma}(t,x) D^\gamma_\g, 
		\end{align*}
where $m$ is a positive integer, $m_j \in \IN_0 \cap [0, 2m)$, $k_j\in \IN_0$ for $j=0, \cdots, m$.
The symbols $D$, respectively $D_\g$ mean $-i\nabla$, respectively $-i\nabla_\g$, where $\nabla$ denotes the gradient in $G$ and $\nabla_\g$ the surface gradient on $\g$. 
We assume that all boundary operators $\cB_j$ and at least one $\cC_j$ are non-trivial. 
The order $k_j$ is defined by $-\infty$ when $\cC_j=0$. 
The unknown functions $u(t,x), \rho(t,x)$ belongs to Hilbert spaces $E$ and $F$.
Note that the case $E=F=\IC^N(N\in \IN)$ is allowed. 
For the coefficients of the above differential operators, $a_\alpha(t,x), b_{j\beta}(t,x)\in \cB(E)$, $c_{j\gamma}(t,x)\in \cB(F, E)$ for $j=1, \cdots, m$, and $b_{0\beta}(t,x)\in \cB(E, F)$ and $c_{0\gamma}(t,x) \in \cB(F)$. 
Let $1<p,q< \infty$. 
We would like to find the maximal $L_p$-$L_q$ regularity solutions, i.e. 
	\begin{align*}
		u \in Z_u 
		:= L_p(J; W^{2m}_q(G; E)),
	\end{align*}
 then we should assume 
	\begin{align*}
		f \in X 
		:= L_p(J; L_q(G; E)). 
	\end{align*}
Since we expect the regularity of $g_j$ is the same as $\cB_j u$, 
	\begin{align*}
		& g_0 \in Y_0 
		:= L_p(J; W^{2m\kappa_0}_{q}(\g; F)),  \\
		& g_j \in Y_j 
		:= L_p(J; W^{2m\kappa_j}_{q}(\g; E)) \quad (j = 1, \cdots , m)
	\end{align*}
with 
	\begin{align*}
		\kappa_j 
		:= 1- \frac{m_j}{2m} - \frac{1}{2mq} \quad (j = 0, \cdots, m)
	\end{align*}
from the trace theorem. 
Thus, the solution class which $\rho$ belongs to should be 
	\begin{align*}
		\rho &\in  W^1_p(J; W^{2m\kappa_0}_{q}(\g; F)) 
		\cap \bigcap_{j=0}^m L_p(J; W^{k_j + 2m\kappa_j}_{q} (\g; F))\\
		&= W^1_p(J; W^{2m\kappa_0}_{q}(\g; F)) 
		\cap L_p(J; W^{l + 2m\kappa_0}_{q} (\g; F)) \\
		& =: Z_\rho, 
	\end{align*}
from the differential structure of the equation (\ref{eq_qs}), where $l:=\max_{j=0, \cdots, m} l_j$ with $l_j = k_j -m_j +m_0$.
We always assume $l \geq 0$ in this paper.
It can be expected by the trace theorem that
	\begin{align*}
		\rho_0 \in \pi Z_\rho 
		:= B^{l(1-1/p)+ 2m\kappa_0}_{qp}(\g; F).
	\end{align*}
Under these settings and assumptions (E), (SA), (SB), (SC), (LS) and (ALS) introduced later,  we shall show the solution operator is an isomorphism between the data $(f, \{g_j\}_{j=0}^m, \rho_0) \in X\times \prod_{j=0}^m Y_j \times \pi Z_\rho$ and the solution $(u, \rho) \in Z_u \times Z_\rho$. 
	
First we assume normal ellipticity of $\cA$ as usual. 
The subscript $\#$ denotes the principal part of the corresponding operator, e.g. $\cA_\#(t,x,D) =\sum_{|\alpha|=2m} a_\alpha(t,x) D^\alpha$. 

\flushleft{\textbf{(E)}} (Ellipticity of the interior symbol)
For all $t\in J$, $x\in \cl{G}$ in the case $G$ is a bounded domain, $x\in \cl{G}\cup\{\infty\}$ in the case $G$ is an exterior domain, and for all $\xi \in \IR^n$ satisfying $|\xi|=1$, we assume normal ellipticity for $\mathcal{A} (t, x, \xi)$ with an angle less than $\pi / 2$, and thus
\[\sigma(\cA_\# (t,x,\xi)) \subset \IC_+:=\{z\in\IC\mid \Re z >0\}. \]
Here $\sigma(\cA_\# (t,x,\xi))$ denotes the spectrum of the bounded operator $\cA_\# (t,x,\xi) \in \cB(E)$. 

Next, we introduce conditions of smoothness to the coefficients of $\mathcal{A}$, $\mathcal{B}_j$ and $\mathcal{C}_j$.
These conditions allow us to use localization and perturbation argument.
\flushleft{\textbf{(SA)}} For $|\alpha|=k\le 2m-1$, there exists $r_\alpha \ge q$ with $\frac{n}{r_\alpha}\le 2m-k$ such that 
	\begin{align*}
		a_\alpha \in L_\infty(J; L_{r_\alpha}(G; \cB(E))). 
	\end{align*}
For $|\alpha|=2m$, assume 
	\begin{align*}
		a_\alpha \in BUC(J\times \cl{G}; \cB(E)). 
	\end{align*}
In the case $G$ is exterior domain, we impose the condition that the asymptotic state at infinity $a_{\alpha}(t,\infty):=\lim_{|x|\to \infty, x\in G}a_\alpha(t,x)$ exists and is bounded uniformly with respect to $t\in J$ for all $|\alpha|=2m$.
	
\flushleft{\textbf{(SB)}} Let $\cE_0:=\cB(E, F)$ and $\cE_j:=\cB(E)$ for $j=1, \cdots, m$. 
For each $j=0, \cdots, m$ and $|\beta|=k \le $ $ m_j -1$, there exist $s_{j\beta}, r_{j\beta} \ge q$ with $\frac{n-1}{s_{j\beta} }\leq m_j - k$, $\frac{n-1}{r_{j\beta}} \leq 2m - k - 1/q$ such that 
	\begin{align*}
		b_{j\beta} \in L_\infty(J; (L_{s_{j\beta}} 
		\cap B^{2m\kappa_j}_{r_{j\beta}, q})(\g; \cE_j)). 
	\end{align*}
For $|\beta|=m_j$, assume 
	\begin{align*}
		b_{j\beta} \in BUC(J\times \g; \cE_j). 
	\end{align*}
	
\flushleft{\textbf{(SC)}} Let $\cF_0:=\cB(F)$ and $\cF_j:=\cB(F, E)$ for $j=1, \cdots, m$. 
For each $j=0, \cdots, m$ and $|\gamma|= $ $k\le k_j -1$, there exist $t_{j\gamma}, \tau_{j\gamma}\ge p$ and $s_{j\gamma}^c, r_{j\gamma}^c \ge q$ with $\frac{l}{t_{j\gamma}} + \frac{n-1}{s_{j\gamma}^c} \le l - k + m_j - m_0$ and $\frac{l}{\tau_{j\gamma}} + $ $ \frac{n-1}{r_{j\gamma}^c} \le l - k + 2m\kappa_0$ such that 
	\begin{align*}
	c_{j\gamma} \in L_{t_{j\gamma}}(J; L_{s_{j\gamma}^c}(\g; \cF_j)) \cap L_{\tau_{j\gamma}}(J; B^{2m\kappa_j}_{r_{j\gamma}^c, q}(\g; \cF_j))
	\end{align*}
For $|\beta|=k_j$, assume 
	\begin{align*}
		c_{j\gamma} \in BUC(J\times \g; \cF_j). 
	\end{align*}

The following two conditions are needed to get the formula of solution operator and ensure their boundedness. 

{\bf (LS)}(Lopatinskii--Shapiro conditions) For each fixed $t \in J$ and $x\in \g$, we freeze the coefficients of differential operator at $(t,x)$. 
We rewrite the equations (\ref{eq_qs}) in coordinates associated with $x$ so that the positive part of $x_n$-axis has the direction of the inner normal at $x$ after a transformation and a rotation.
For all $\eta > 0$, $(\lambda, \xi')\in (\DomainOfLambda \times\IR^{n-1} )\setminus \{ (0,0) \}$ ($\theta > \pi / 2$) and $\{h_j\}_{j=0}^m\in F\times E^m$, the ODEs on the half line $\IR_+=(0,\infty)$ given by 
	\begin{align}\label{eq_ls}
		\left \{
			\begin{array}{rll}
				\eta v (y)
				+ \cA_\#(t,x,\xi', D_y) v(y) 
				&= 0 \quad &(y>0), \\
				\cB_{0\#}(t,x,\xi', D_y) v(0) 
				+(\lambda + \cC_{0\#}(t,x, \xi')) \sigma 
				&= h_0, \quad 
				& \,  \\
				\cB_{j\#}(t,x,\xi', D_y) v(0) 
				+ \cC_{j\#}(t,x,\xi') \sigma 
				&= h_j \quad 
				&(j = 1, \cdots, m)
			\end{array}
		\right.
	\end{align}
admit a unique solution $(v, \sigma) \in C_0^{2m}( \Real_+; E) \times F$, where 
	\begin{align*}
		C^{2m}_0 ( \Real_+; E) 
		= \left \{ v \in C^{2m} ( \Real_+; E) \, ; \, \lim_{y \rightarrow \infty} v (y) = 0 \right \}. 
	\end{align*}
To obtain the maximal $L_p$-$L_q$ regularity, we need another type of Lopatinskii--Shapiro condition which ensures boundedness of the symbol of the solution operator. 

\flushleft{\textbf{(ALS)}} (Asymptotic Lopatinskii--Shapiro conditions) For each fixed $t \in J$ and $x\in \g$ we rewrite the equations (\ref{eq_qs}) by the same way as above.
For all  $\eta  > 0$, $\xi \in \Real^{n- 1}$ and $\{h_j\}_{j=1}^m\in F\times E^m$,
	\begin{align*}
			\begin{array}{rll}
				\eta v (y)
				+ \cA_\#(t,x,\xi', D_y) v(y) 
				&= 0 \quad &(y>0), \\
				\cB_{j\#}(t,x,\xi', D_y) v(0) 
				&= h_j \quad 
				&(j = 1, \cdots, m)
			\end{array}
	\end{align*}
admit a unique solution $v \in C_0^{2m}(\Real_+; E)$. 
For all $(\lambda, \xi')\in (\DomainOfLambda\cup\{0\}) \times \mathbb{S}^{n-2}$ ($\theta > \pi / 2$), all $\{h_j\}_{j=0}^m\in F\times E^m$ the ordinary differential equations in $\IR_+$ given by 
	\begin{align*}
		\begin{array}{rll}
			\cA_\#(t,x,\xi', D_y) v(y) 
			&= 0 \quad 
			&(y>0), \\
			\cB_{0\#}(t,x,\xi', D_y) v(0) 
			+(\lambda + \delta_{l, l_0}\cC_{0\#}(t,x, \xi')) \sigma 
			&= h_0, \quad 
			& \, \\
			\cB_{j\#}(t,x,\xi', D_y) v(0) 
			+ \delta_{l,l_j}\cC_{j\#}(t,x,\xi') \sigma 
			&= h_j \quad 
			&(j = 1, \cdots, m)
		\end{array}
	\end{align*}
admit a unique solution $(v, \sigma) \in C_0^{2m}(\Real_+; E) \times F$. 
Here $\IS^{n-2}:=\{\xi' \in \IR^{n-1} ; |\xi'|=1\}$ and $\delta_{i,j}$ is the Kronecker delta, $\delta_{i,j} = 1$ for $i=j$ and $\delta_{i,j}=0$ for $i\neq j$. 
Moreover, we assume the following elliptic equations. 
For $\xi'\in \IS^{n-2}$, 
	\begin{align*}
		\cA_{\#}(t, x, \xi', D_y) v(y) 
		&= 0 \quad (y>0), \\
		\cB_{j\#}(t, x, \xi', D_y) v(0) 
		&= h_j \quad (j=1, \cdots m)
	\end{align*}
admit a unique solution $v \in C_0^{2m}(\Real_+; E)$, respectively. 
We are now in the position to state our main results. 

\begin{theorem}\label{thm_main}
Let $J=[0, T]$, $G\subset \IR^n$ be a domain with a compact boundary $\g=\partial G$ of class $C^{2m + l -m_0}$, $1<p, q<\infty$ and $E$ and $F$ be Hilbert spaces. 
Assume assumptions $(E)$,  $(SA)$,  $(SB)$,  $(SC)$,  $(LS)$ and $(ALS)$ hold.
Then, there exist positive constants $\eta_0$, $C$ and $C_T$, if $\eta \geq \eta_0$, for
	\begin{align*}
		(f, \{g_j\}_{j=0}^m, \rho_0) \in X \times \prod_{j=0}^m Y_j \times \pi Z_\rho,
	\end{align*} 
(\ref{eq_qs}) admits a unique solution $(u, \rho) \in Z_u \times Z_\rho$ such that
		\begin{align*}
			\Vert
				u
			\Vert_{Z_u}
			+ \Vert
				\rho
			\Vert_{Z_\rho}
			\leq C \Vert
				f
			\Vert_{X}
			+ C \sum_{j= 0}^m \Vert
				g_j
			\Vert_{Y_j}
			+  C_T \Vert
				\rho_0
			\Vert_{\pi Z_\rho}.
		\end{align*}
\end{theorem}

\section{Preliminaries}\label{Sec: Pre}

In this section, notation, notion, basic tools of vector-valued harmonic analysis are introduced.
For Banach spaces $X$ and $Y$, $\mathcal{B} (X ; Y)$ denotes the set of bounded linear operators from $X$ to $Y$.
$H^\infty (\Sigma_\phi )$ denotes the set of bounded holomorphic functions on a sector 
	\begin{align*}
		\Sigma_\phi : = \left \{ r e^{i \psi} \in \Complex \setminus \{ 0 \}  \, ; r > 0, \, \abs{\psi} < \phi \right \}.
	\end{align*}
For a Banach space $X$, $H^\infty (\Sigma_\phi ; X)$ is the set of $X$-valued bounded holomorphic functions on $\Sigma_\phi$ for $0< \phi < \pi$ equipped with the norm 
	\begin{align*}
		\Vert f \Vert_{H^\infty (\Sigma_\phi)} 
		: = \sup_{\lambda \in \Sigma_\phi} \abs{f (\lambda)}.
	\end{align*}
$L_p(\Omega ; X)$ ($1 \leq p \leq \infty$) and $W^s_q (\Omega ; X)$ ($s \in \Real, \, 1 \leq q \leq \infty$) are the $X$-valued Lebesgue space and the Sobolev space on $\Omega$.
$\Fourier$ and $\FourierInverse$ are Fourier transform and its inverse transform, respectively.
Especially, we denote $\Fourier_{x^\prime}$ by the partial Fourier transform with respect to $x^\prime$-variable.
$\Laplace$ and $\LaplaceInverse$ are Laplace transform and its inverse transform, respectively.
	
	\begin{definition}
		A Banach space $X$ is said to be of class $\HT$ if the Hilbert transform $H$ defined by 
		\begin{align*}
			Hf(t) 
			:=\frac{1}{\pi} \lim_{R\to\infty} \int_{R^{-1}\le |s| \le R} f(t-s) \frac{ds}{s}
		\end{align*}
		is bounded on $L_p(\IR; X)$ for some $p\in(1, \infty)$. 
		When $X$ is of the class $\HT$, then $L_p(J; X)$ is also of class $\HT$. 
	\end{definition}
	
	\begin{definition}
Let $X$ and $Y$ be Banach spaces. 
A family of operators $\cT \subset \cB(X ; Y)$ is said to be $\cR$-bounded, if there exists a constant $C>0$ and $p\in[1, \infty)$ such that, for each positive integer $N$, $\{T_i\}_{i=1}^N \subset \cT$, $\{x_i\}_{i=1}^N \subset X$ and for all independent symmetric $\{-1, 1\}$-valued random variables $\eps_i$ on a probability space $(\Omega, \cA, \mu)$, the inequality 
		\begin{align} \label{eq_R_bounded}
			\left|
				\sum_{i=1}^N \eps_i T_i x_i
			\right|_{L_p(\Omega; Y)}
			\le C \left|
				\sum_{i=1}^N \eps_i x_i
			\right|_{L_p(\Omega; X)}
		\end{align}
holds. 
We denote by $\mathcal{R} \mathcal{T}$ the infimum constant of $C$ which (\ref{eq_R_bounded}) holds.
	\end{definition}
It is known that if (\ref{eq_R_bounded}) holds for some $p\in[1, \infty)$, then (\ref{eq_R_bounded}) holds for all $p\in[1, \infty)$.
Note that uniformly bounded family of operators on Hilbert spaces is always $\mathcal{R}$-bounded.

	\begin{definition}
		A Banach space $X$ is said to have property $(\alpha)$ if there exists a constants $C>0$ such that 
		\begin{align*}
			\left|
				\sum_{i,j=1}^N \alpha_{ij}\eps_i \eps_j^{\prime} x_{ij}
			\right|_{L_2(\Omega\times\Omega'; X)}
			\le C \left|
				\sum_{i,j=1}^N \eps_i \eps_j^{\prime} x_{ij}
			\right|_{L_2(\Omega\times\Omega'; X)}
		\end{align*}
for all $\alpha_{ij}\in\{-1, 1\}$, $\{x_{ij}\}_{i,j=1}^N\subset X$, positive integer $N$, and all symmetric independent $\{-1, 1\}$-valued random variables $\eps_i$ (respectively $\eps_j^\prime$) on a probability space $(\Omega, \cA, \mu)$ (respectively $(\Omega', \cA', \mu')$). 
$\HT(\alpha)$ denotes the class of Banach spaces which belong to $\HT$ and have property $(\alpha)$.
	\end{definition}
Note that Hilbert space is of the class $\HT(\alpha)$ and all closed subspaces of $L_p (G)$ have property $(\alpha)$.
	
	\begin{proposition}[Operator-valued Fourier multiplier theorem of Lizorkin type](see \cite{PrussSimonett2016})
Let $1 $ $< p < \infty$, $X$ and $Y$ be Banach spaces of the class $\mathcal{HT} (\alpha)$.
Let $\mathcal{M} \subset C^n (\Real^n \setminus \{ 0 \} ; \mathcal{B} (X ; Y))$ be a family of multipliers such that
		\begin{align*}
			\mathcal{R} \left \{
				\xi^\alpha \partial_\xi^\alpha m (\xi) 
				\, : \, \xi \in \Real^n \setminus \{ 0 \}, \, \alpha \in \{0, 1 \}^n, \,  m \in \mathcal{M}
			\right \}
			= : C_L
			< \infty.
		\end{align*}
Then $\FourierInverse m \Fourier \in \mathcal{B} \left( L^p (\Real^n ; X) \,	; \, L^p (\Real^n ; Y) \right)$.
Moreover,
		\begin{align*}
			\mathcal{R} \left \{
				\FourierInverse m \Fourier \,
				; \, m \in \mathcal{M}
			\right \}
			\leq C C_L,
		\end{align*}
for some constant $C = C (p, n,  X ,Y)$.
	\end{proposition}
	
We define a class of holomorphic functions vanishing at the origin and infinity by
	\begin{align*}
		H^\infty_0 (\Sigma_\phi)
		= \left \{
			f \in H^\infty (\Sigma_\phi) \,
			; \, \abs{f (\lambda)} \leq C \abs{\chi (\lambda)}^\varepsilon \, \, \mathrm{for \, \, some}\, \, C > 0, \, \, \varepsilon > 0
		\right \},
	\end{align*}
where $1 < \phi < \pi$ and $\chi (\lambda) = \lambda / ( 1 + \lambda)^2$.
Let $0<\phi_A < \phi < \pi$ and $A$ be a sectorial operator  with spectral angle $\phi_A$ and $f \in H_0^\infty (\Sigma_\phi)$.
We define $f (A)$ via the Cauchy formula
	\begin{align*}
		f (A)
		:= \frac{1}{2 \pi i} \int_{\partial \Sigma_\phi}
			(\lambda - A)^{-1} f (\lambda)
		d \lambda.
	\end{align*}
It is called that a sectorial operator $A : D (A) \subset X \rightarrow Y$ with spectral angle $\phi_A$ have a bounded $H^\infty$-calculus if there exists a constant $C > 0$
	\begin{align} \label{def_H_infty_calculus_estimate}
		\Vert
			f (A)
		\Vert_{\mathcal{B}(X ; Y)}
		\leq C \Vert
			f
		\Vert_{H^\infty (\Sigma_\phi)}
	\end{align}
holds for all $f \in H^\infty_0 (\Sigma_\phi)$, ($\phi > \phi_A$).
Sectorial operators satisfying (\ref{def_H_infty_calculus_estimate}) have an extended calculus $f (A)$ for $f \in H^\infty (\Sigma_\phi)$ by the canonical way, and this extension is uniquely determined.
	\begin{definition}
Let $X$ be a Banach space.
Let $0<\phi_A <\pi$ and $A$ be a sectorial operator on $X$ with spectral angle $\phi_A$ admitting a bounded $H^{\infty}$-calculus.
$A$ is said to have a $\mathcal{R}$-bounded $H^{\infty}$- calculus if 
		\begin{align} \label{eq_R_bounded_H_infty_calculus}
			\LongSet{
				h (A)
			}{
				h \in H^\infty (\Sigma_\phi ), |	h |_{ H^\infty (\Sigma_\phi )} \leq 1
			}
		\end{align}
is $\mathcal{R}$-bounded for some $\phi \geq \phi_A$.
Such an operator is denoted by $A \in \mathcal{R}H^\infty (X)$.
We denote by $\phi_{\mathcal{R} H^\infty}$ the infimum of $\phi$ which (\ref{eq_R_bounded_H_infty_calculus}) holds.
	\end{definition}
Let us introduce the Kalton--Weis theorem, which gives a sufficient condition for boundedness of joint functional calculus and is used to show boundedness of solution operator in this paper, see e.g. \cite{DoreVenni2005}, \cite{KaltonWeis2001}, \cite{KunstmannWeis2004} and \cite{PrussSimonett2016}.
	
	\begin{lemma}[Kalton--Weis Theorem](see \cite{KaltonWeis2001},\cite{DoreVenni2005}, \cite{KunstmannWeis2004} and \cite{PrussSimonett2016}) \label{lem_Kalton_Weis}
Let $X$ be a Banach space of the class $\mathcal{HT}(\alpha)$, $A$ be a sectorial operator with spectral angle $\phi_A$ admitting a bounded $H^{\infty}$-calculus and $\mathcal{F}$ be a family of operators satisfying $\mathcal{F} \subset H^\infty (\Sigma_\phi ; \mathcal{B} (X))$ for $\phi > \phi_A$.
Assume each $F \in \mathcal{F}$ commute with the resolvent of $A$, i.e. $F (\lambda) (\mu - A)^{-1} = (\mu - A)^{-1} F (\lambda)$, and
		\begin{align*}
			\LongSet{ F (z) }{ z \in \Sigma_\phi, F \in \mathcal{F}}
		\end{align*} 
is $\mathcal{R}$-bounded.
Then there exist a constant $C > 0$ such that
		\begin{align*}
			\mathcal{R} \left(
				\mathcal{F} (A)
			\right)
			\leq C \mathcal{R} \LongSet{
				F (z)
			}{
				z \in \Sigma_\phi, F \in \mathcal{F}
			}.
		\end{align*}
	\end{lemma}
Lemma \ref{lem_Kalton_Weis} also implies each operator admitting bounded $H^\infty$-calculus belongs to $\mathcal{R}H^\infty$ provided that $X$ is of class $\mathcal{HT} (\alpha$). 

\section{Solvability in the Maximal Regularity Space} \label{Sec: solvability}

\subsection{Reduction to $f= 0$ and $\rho_0=0$}\label{subsec: Reduction}
	
We first consider our problem on the half space $\Real^{n}_+$ and assume the differential operators have constant coefficients without lower order.
Let $E_{\IR^n}$ be the zero extension operator from $L_p(J; L_q(\IR^n_+; E))$ to $L_p(J; L_q(\IR^n; E))$.
It follows from the Mikhlin theorem that there exist a unique solution $u_\ast$ to 
	\begin{align*}
		\eta u 
		+ \cA u 
		= E_{\IR^n}f \quad (t \in J, \,  x \in \IR^n),
	\end{align*}
for $f \in L_p (J ; L_q (\Real^n_+ ; E))$ such that
		\begin{align*}
			\eta \Vert
				u_\ast
			\Vert_{L_p (J ; L_q (\Real^{n} ; E))}
			+ \Vert
				u_*
			\Vert_{ L_p (J; W^{2m}_q (\Real^n; E))}
			\leq C \Vert
				f
			\Vert_{L_p ( J ; L_q (\Real^n_+ ; E) )}.
		\end{align*}
Let $\rho_0 \in B^{2m \kappa_0 + l (1 - 1 / p)}_{q, p} ( \Real^{n-1} ; F )$.
Then we also find a unique solution $\rho_\ast$ via maximal regularity of  to $(\eta - \Delta)^{l/2}$ that
	\begin{align*}
		\begin{array}{rll}
			\partial_t \rho 
			+ (\eta - \Delta)^{l/2} \rho 
			& = 0 \quad 
			&(t \in \Real_+, x \in \IR^{n-1}), \\
			\rho(0) 
			&= \rho_0 \quad 
			&(x \in \IR^{n-1})
		\end{array}
	\end{align*} 
such that
	\begin{align*}
		\Vert
			\rho_\ast
		\Vert_{
			W^1_p (\Real_+ ; W^{2 m \kappa_0}_q (\Real^{n - 1} ; F)) 
			\cap L_p(\IR_+; W^{l + 2m\kappa_0}_{q}(\IR^{n-1}; F))
		}
		\leq C \Vert
			\rho_0
		\Vert_{B^{2m \kappa_0 + l (1 - 1 / p)}_{q, p} (\Real^{n-1} ; F)}.
	\end{align*}
For the solution $(u, \rho)$ to (\ref{eq_qs}), if we put $(\tilde{u}, \tilde{\rho}) := (u - u_\ast, \rho - \rho_\ast)$, then $(\tilde{u}, \tilde{\rho})$ satisfies
	\begin{align} \label{eq_reduced_qs}
		\begin{array}{rll}
			\eta \tilde{u}
			+ \cA \tilde{u} 
			&= 0 \quad 
			&(t\in J,~ x\in G), \\
			\partial_t \tilde{\rho}  
			+\cB_0 \tilde{u} + \cC_0 \tilde{\rho} 
			&= g_0 
			- (\partial_t \rho_\ast  
			+\cB_0 u_\ast + \cC_0 \rho_\ast) \quad 
			&(t\in J,~ x\in \g), \\
			\cB_j \tilde{u} 
			+ \cC_j \tilde{\rho} 
			&= g_j 
			-  (
				\cB_j u_\ast 
				+ \cC_j \rho_\ast
			)\quad 
			&(t\in J,~ x\in \g,~j = 1, \cdots, m), \\
			\tilde{\rho}(0,x) 
			&= 0 \quad 
			&(x\in \g).
		\end{array}
	\end{align}
Note that $g_0 - (\partial_t \rho_\ast  +\cB_0 u_\ast + \cC_0 \rho_\ast) \in Y_0$ and $g_j -  (\cB_j u_\ast + \cC_j \rho_\ast) \in Y_j$. 
Conversely, the solution of the original equations is given by $(u, \rho) := (\tilde{u} + u_\ast, \tilde{\rho} + \rho_\ast)$.
Thus, it suffice to consider the case of $f = 0$ and $\rho_0 = 0$ from now on. 

\subsection{Partial Fourier transform and solution formula on the half space}\label{subsec: FT}
	
We continue to consider the case of the half space and assume  differential operators having constant coefficients without lower order terms.
Assume that $(u, \rho)$ are solutions to (\ref{eq_qs}) with $f = 0$ and $\rho_0 = 0$.
Put
	\begin{align*}
		v = \Laplace_t \Fourier_{x^\prime} u,
		\quad \sigma = \Laplace_t \Fourier_{x^\prime} \rho.
	\end{align*}
Then $(v, \rho)$ satisfy
	\begin{align} \label{eq_ODE}
		\left\{
			\begin{aligned}
				\eta v + \cA(\xi', D_y) v 
				&= 0, \quad 
				&\,  \\
				\cB_0(\xi', D_y) v(0) 
				+ (\lambda + \cC_0(\xi')) \sigma 
				&= h_0, &\\
				\cB_j(\xi', D_y) v(0) + \cC_j(\xi') \sigma &= h_j \qquad 	
				&(j=1,\cdots, m),
			\end{aligned}
		\right.
	\end{align}
where $h_j = \Laplace_t \Fourier_{x^\prime} g_j$ for $j = 0, \cdots , m$.
	
The Lopatinskii--Shapiro condition (LS) ensures, for each $(\lambda, \xi') \in (\DomainOfLambda\times\IR^{n-1} ) \setminus \{(0,0)\}$ ($\theta > \pi /2$) and for any  $\{h_j\}_{j=0}^m \in F \times E^m$, there exists a unique solution 
	\begin{align*}
		v \in C^{2m}_0 (\IR_+; E),
		\quad \sigma \in F
	\end{align*}
to (\ref{eq_ODE}). 
We derive the solution operator of Fourier multiplier type and show its  boundedness. 
As \cite{DenkHieberPruss2003,Denk-Pruss-Zacher2008,PrussSimonett2016}, we construct the solution formula. 
By definition of $\mathcal{A}$ and $\mathcal{B}_j$, 
	\begin{align*}
		\cA(\xi' , D_y) 
		= \sum_{k=0}^{2m} a_k(\xi') D_y^{2m-k}, \quad 
		\cB_j(\xi' , D_y) 
		= \sum_{k=0}^{m_j} b_{jk}(\xi') D_y^{m_j-k}, 
	\end{align*}
where $a_k(\xi')$ and $b_{jk}(\xi')$ are homogeneous of degree $k$. 
Set
	\begin{align*}
		\mu 
		= (\eta + \abs{\xi^\prime}^{2m})^{1/2m}, \quad 
		b 
		= |\xi'|/\mu , \quad 
		\zeta 
		= \xi'/ \mu , \quad 
		a 
		= \eta/\mu^{2m}
	\end{align*}
and $w:= (w_1, \cdots, w_{2m})^T$ for
	\begin{align*}
		w_k 
		:= \left(
			\frac{1}{\mu} D_y
		\right)^{k-1}v \quad (k=1, \cdots, 2m). 
	\end{align*}
Note that $(\mu, \zeta, a) \in [\eta^{1/2m}, \infty) \times B_{\Real^{n-1}} (0 \, ; 1) \times (0 ,1] $, where $B_{\Real^d}(c;r)$ is the $d$-dimensional open ball with center $c$ and radius $r$.
Then the first equation of (\ref{eq_ODE}) is equivalent to
	\begin{align} \label{eq_ODE_w}
		D_y w 
		= \mu A_0(\zeta,  a) w,
	\end{align}
where
	\begin{align*}
		A_0(\zeta, a) 
		:= \begin{pmatrix}
			0 &  I & 0 & \cdots & 0 \\
			0 & 0 & I & \cdots & 0 \\
			\vdots & \vdots & \vdots & \vdots & \vdots \\
			0 & 0 & 0 & 0 & I \\
			c_{2m} &  c_{2m-1} &  \cdots &  c_2 &  c_1 &  
		\end{pmatrix}, 
		\end{align*}
and
	\begin{align*}
		c_j 
		&:= c_j (\zeta) := - a_0^{-1} a_j(\zeta) \qquad (j=1,\cdots, 2m-1),\\
		c_{2m} 
		&:= c_{2m}(\zeta, a) := - a_0^{-1}(a_{2m}(\zeta) + a).
	\end{align*}
Actually, it follows from the first equation of (\ref{eq_ODE}) and definition of $w_j$ that
	\begin{align*}
		\quad (1/\mu) D_y w_{2m} 
		= - a_0^{-1}\left(
			a_{2m} (\xi'/\mu) 
			+ \eta/\mu^{2m} 
		\right) w_1 
		- \sum_{k=2}^{2m} a_0^{-1} a_{2m-k+1} (\xi'/\mu) w_k. 
	\end{align*}
Thus, we find (\ref{eq_ODE_w}) from the definition of $w$.
Moreover, (\ref{eq_ODE_w}) implies
	\begin{align*}
		w(y) 
		:= e^{\mu i A_0(\zeta, a) y} w (0) \quad (y \ge 0).
	\end{align*}
We write $w(0) = w|_{y = 0}$ for simplicity.
The functions $w (y)$ have to be determined so that tends to zero at infinity.
This is guaranteed by 
	\begin{align*}
		P_+(\zeta, a) w_0 
		= 0,
	\end{align*} 
where $P_+( \zeta, a) \in \cB(E^{2m})$ is the associated positive spectral projection with $i A_0(\zeta, a)$. 
Note that each spectrum of $i A_0(\zeta, a))$ do not lie on the imaginary axis and $P_+$ is holomorphic and bounded uniformly in $( \zeta, a)$ by the Lopatinskii--Shapiro condition since $( \zeta, a)$ run a compact set away from $(0, 0)$.
Supremum of real part of negative spectrum of $i A (\zeta, a)$ is less than zero and infimum of real part of positive spectrum is larger than zero, this facts imply $e^{\mu i A_0 (\zeta, a) y} w (0) \rightarrow 0$ for $w (0)$ satisfying $P_+ (\zeta, a) w (0) = 0$ as $ y \rightarrow \infty$.
See \cite{DenkHieberPruss2003} and \cite{PrussSimonett2016} for details of the above discussion.
	
Let $w := \Fourier_{x^\prime} \Laplace_{t} h$ for $h \in L_p (J; W^{2m - 1 / p}_q (\Real^{n-1}; E^{2m}))$.
Define the canonical extension of functions from the boundary to the half space by 
	\begin{align}
		\mathcal{T} h
		:= \LaplaceInverse_{\lambda} \FourierInverse_{\xi^\prime} 
		\left[
			\mu^{2m} e^{\mu i A_0(\zeta, a) y} \left(
				I - P_+ (\zeta, a)
			\right)
			w
		\right].
	\end{align}
Boundedness of this extension operator is ensured by the following
	\begin{proposition} \label{prop_boundedness_of_canonical_extension}
Let $1 < p, q < \infty$, $\eta \in \Sigma_\varphi$ for small $\varphi > 0$ and $J = \Real_+$.
Then there exists a constant $C>0$ such that
		\begin{align}
			\Vert
				\mathcal{T} h
			\Vert_{L_p (J; L_q (\Real^{n}_+; E^{2m}))}
			\leq C \Vert
				h
			\Vert_{L_p (J; W^{2m - 1 / p}_q (\Real^{n-1}; E^{2m}))}.
		\end{align}
for $h \in L_p (J; W^{2m - 1 / p}_q (\Real^{n-1}; E^{2m}))$.
	\end{proposition}
	\begin{proof}
See \cite{DenkHieberPruss2003}, section 7.
	\end{proof}
Put $w^0 = w (0)$.
Let us continue to seek the solution formula of $(w^0, \sigma)$.
Since  
	\begin{align*}
		\mathcal{B}_j v
		= \sum_{k=0}^{m_j} b_{jk}(\xi') \mu^{m_j - k} w_{m_j - k + 1}
		= \sum_{k=0}^{m_j} b_{jk}(\zeta) \mu^{m_j} w_{m_j - k + 1}
	\end{align*}
and
	\begin{align*}
		\mathcal{C}_j (\xi^\prime) \sigma
		= \mathcal{C}_j (\zeta) \mu^{k_j} \sigma
	\end{align*}
the second and the third equations of (\ref{eq_ODE}) are equivalent to 
	\begin{align}
		B_0 (\zeta) w^0 
		+ \left\{ \lambda \mu^{-m_0} 
		+ \cC_0 (\zeta) \mu^{-m_0 + k_0} \right\} \sigma 
		&= \mu^{-m_0} h_0, \nonumber\\
		B_j (\zeta) w^0 
		+ \cC_j (\zeta) \mu^{-m_j + k_j} \sigma 
		&= \mu^{-m_j} h_j \quad (j=1, \cdots, m), \label{eq_LS}\\ 
		P_
		+ (\zeta, a) w^0 
		&= 0, \nonumber
	\end{align}
where $B_j(\zeta) := (b_{jm_j}(\zeta), \cdots, b_{j0}, 0, \cdots,0)$ for $j=0, \cdots, m$. 
Note that by the assumptions on (E) and (LS) the above equations (\ref{eq_LS}) admit a unique solution 
	\begin{align*}
		(w^0, \sigma) \in E^{2m} \times F
	\end{align*}
for each $(\lambda, \xi) \in (\DomainOfLambda\times\IR^{n-1}) \setminus\{(0,0)\} $ ($\theta > \pi / 2$) and $\{h_j\}_{j=0}^m \in F \times E^m$.  
Introducing 
	\begin{align*}
		\sigma^0 
		:= (\lambda + \mu^l) \mu^{-m_0} \sigma, \quad
		h 
		:= (h_j^0)_{j=0}^m := (\mu^{-m_j} h_j)_{j=0}^m,
	\end{align*} 
we rewrite (\ref{eq_LS}) into
	\begin{align} 
			\begin{array}{rlr}
			\displaystyle B_0 (\zeta) w^0
			+ \frac{
				\lambda 
			+ \cC_0 (\zeta) \mu^{l_0}
			}{
				\lambda + \mu^l
			} \sigma^0 
			&= h_0^0,
			& \, \\ [2ex]
			\displaystyle B_j (\zeta) w^0 
			+ \frac{
				\cC_j (\zeta) \mu^{l_j}
			}{
				\lambda + \mu^l
			} \sigma^0 
			&= h_j^0 
			& (j=1, \cdots, m), \\ [2ex]
			P_+ (\zeta, a) w^0 
			&= 0.
			& \,  
		\end{array}
	\end{align}
Thus, it follows
	\begin{align} \label{eq_LS_w0_sigma0}
		\begin{array}{rlr}
			\displaystyle B_0 (\zeta) w^0
			+ \frac{
				\nu
			+ \cC_0 (\zeta)  \eta^{- (l_0 - l)  / 2m} \tilde{a}^{ l - l_0 }
			}{
				\nu + 1
			} \sigma^0 
			&= h_0^0,
			& \, \\ [2ex]
			\displaystyle B_j (\zeta) w^0 
			+ \frac{
				\cC_j (\zeta) \eta^{- (l_j - l)  / 2m} \tilde{a}^{ l - l_j}
			}{
				\nu + 1
			} \sigma^0 
			&= h_j^0 
			& (j=1, \cdots, m), \\ [2ex]
			P_+ (\zeta, \tilde{a}^{2m}) w^0 
			&= 0.
			& \,  
		\end{array}
	\end{align}
for $\nu = \lambda / \mu^l$ and $\tilde{a} = a^{1/2m}$.
We write the solution to (\ref{eq_LS_w0_sigma0}) as 
	\begin{align*}
		w^0 
		:= M_w^0 (\zeta, \tilde{a}, \nu) h, \quad 
		\sigma^0 
		:= M_\sigma^0 (\zeta, \tilde{a}, \nu) h. 
	\end{align*}
Set $Y_E := L_p (J; W^{2m-1/p}_{q}(\IR^{n-1}; E))$ and $Y_F:= L_p(J; W^{2m-1/p}_{q}(\IR^{n-1}; F))$.
	\begin{proposition} \label{prop_boundedness_multiplier_operator}
Let $1 < p, q < \infty$, $\eta > 0$ and $G = \Real^n_+$.
Assume assumptions $(E)$, $(LS)$ and $(ALS)$ hold.
Then, there exist a positive constant $C>0$, it holds that
		\begin{align*}
			\Vert
				\LaplaceInverse_{\lambda} \FourierInverse_{\xi^\prime}
				\left[
					\left( 
						M_w^0, M_\sigma^0
					\right)  \Fourier_{x^\prime} \Laplace_t 
				\right]
			\Vert_{\mathcal{B}
				\left(
					Y_F \times Y^m_E
					; Y^{2m}_E \times Y_F
				\right)
				}
			& \leq C.
		\end{align*}
	\end{proposition}
	\begin{proof}
Analyticity of $\left( M_w^0, M_\sigma^0 \right)$ on a open set $D_\zeta \times D_{\tilde{a}} \times \DomainOfLambda \supset B_{\Real^{n-1}} (0 \, ; 1) \times (0,1]  \times \DomainOfLambda$ ($\theta > \pi / 2$) is guaranteed by (LS) and analyticity of $B_j, \, \cC_j$ for $j = 0 , \cdots , m$.
Boundedness of $(M_w^0, M_\sigma^0)$ is equivalent to the solvability for $\zeta \in \overline{B_{\Real^{n-1}}(0 ; 1)}$, $\tilde{a} \in [0, 1]$ and $ \nu \in \DomainOfLambda \cup \{\infty\} $ ($\theta > \pi / 2$).
The solvability of $(M_w^0, M_\sigma^0)$ in the case of $\mu \neq \infty$ and $\lambda \neq \infty$ is guaranteed by (LS).
We need to control behaviour of $(M_w^0, M_\sigma^0)$ on $\mu$ and $\lambda$ at infinity.
Let us consider the case of $| \mu |\to \infty$ or $|\lambda|\to \infty$. 
We find
	\begin{align*}
			\frac{
			\eta^{- (l_0 - l)  / 2m} \tilde{a}^{ l - l_j }
			}{
				\nu + 1
			}
			\to \left \{
				\begin{array}{crl}
					0 
					&\quad {\rm if}
					& |\lambda|/ \abs{\mu}^l \to \infty, \\
					\frac{\delta_{l, l_j}}{c+1} 
					&\quad \mathrm{if}
					& \lambda/ \mu^l \to c, 
				\end{array}
			\right.
 		\end{align*}
and		
		\begin{align*}
			\frac{
				\nu
			}{
				\nu + 1
			}
			\to \left \{
				\begin{array}{crl}
					1 
					&\quad \mathrm{if} 
					& |\lambda|/ \abs{\mu}^l \to \infty, \\
					\frac{c}{c+1} 
					&\quad \mathrm{if}
					& \lambda/ \mu^l \to c, \\
				\end{array}
			\right.
		\end{align*}
for some $c \in \DomainOfLambda\cup\{0\}$.
Let us consider the case (i) (\textbf{$|\lambda|/ \abs{\mu}^l \to \infty$}).
The limit problem of this case is
		\begin{align}
			\begin{array}{rll}
				B_0 (\zeta) w^0
				+  \sigma^0 
				&= h_0^0,
				& \, \\
				B_j (\zeta) w^0  
				&= h_j^0 
				& (j=1, \cdots, m), \\ 
				P_+ (\zeta, a_\ast) w^0 
				&= 0.
				& \,  
			\end{array}
		\end{align}
for some $a_\ast\in[0,1]$ which is the limit of $\tilde{a}^{2m}$.
If $\mu$ tend to infinity at the same time, i.e. $a_\ast=0$, this system corresponds to the following problem; for all $\{h_j^0\}_{j=0}^m \in F\times E^m$ and for any $\xi'\in \mathbb{S}^{n-2}$, 
		\begin{align*}
			\cA(\xi', D_y) v(y) 
			&= 0 \quad (y>0), \\
			\mathcal{B}_0 (\xi^\prime, D_y) v^0
			+ \sigma^0
			& = h_0^0, \\
			\cB_{j}(\xi', D_y) v^0 
			&= h_j^0 \quad (j=1, \cdots , m),
		\end{align*}
admits a unique solution $v \in C_0^{2m}(\IR_+; E)$, which is guaranteed by the third asymptotic Lopatinskii--Shapiro condition.
On the other hand, if $\mu$ is still finite, i.e. $a_\ast\in(0,1]$, the corresponding problem is given by
		\begin{align*}
			\eta v (y) +
			\cA(\xi', D_y) v(y) 
			&= 0 \quad (y>0), \\
			\mathcal{B}_0 (\xi^\prime, D_y) v^0
			+ \sigma^0
			& = h_0^0, \\
			\cB_{j}(\xi', D_y) v^0 
			&= h_j \quad (j=1, \cdots , m) ,
		\end{align*}
for all $\eta>0$, $\xi^\prime \in \Real^{n-1}$ and $\{h_j^0\}_{j=0}^m \in F\times E^m$.
This problem is solvable by the first asymptotic Lopatinskii--Shapiro condition.
Next we consider case (ii) ($\lambda/\mu^l \to c \in \DomainOfLambda\cup\{0\}$).
In these cases, $a_\ast=0$ and the limit problem is
	\begin{align} \label{eq_asymptotic_LS_lambda_sim_xi}
		\begin{array}{rll} \displaystyle
			\cB_0 (\zeta) w^0 
			+ \frac{c + \delta_{l, l_0}\cC_0(\zeta)}{c+1} \sigma^0 
			&= \, h_0^0, 
			& \, \\ \displaystyle
			\cB_j (\zeta) w^0 
			+ \frac{\delta_{l,l_j}\cC_j(\zeta)}{c+1} \sigma^0 
			&= \, h_j^0
			& (j=1, \cdots , m), \\
			P_+(\zeta, 0) w^0
			&= \, 0.
			& \,
		\end{array}
	\end{align}
To ensure solvability of this problem, it is enough to impose the following the condition; for all $\{h_j\}_{j=0}^m \in F \times E^m$ and for any $\lambda \in \DomainOfLambda$ and $\xi' \in \mathbb{S}^{n-2}$, 
	\begin{align*}
		\cA( \xi', D_y) v(y) 
		&= 0 \quad (y>0), \\
		\cB_{0\#}( \xi', D_y) v_0 
		+ (\lambda + \delta_{l,l_0}\cC_{0\#}( \xi')) \sigma &= h_0, \\
		\cB_{j\#}(\xi', D_y) v_0 
		+ \delta_{l,l_j}\cC_{j\#}( \xi')\sigma 
		&= h_j \quad (j=1, \cdots , m)
	\end{align*}
admits a unique solution $(v, \sigma) \in C^{2m}_0(\IR_+; E) \times F$.
This condition is nothing but the second asymptotic Lopatinskii--Shapiro condition.
We find from the above discussion that $M_w^0$ and $M_\sigma^0$ is bounded holomorphic on $D_\zeta \times D_a \times \DomainOfLambda$ ($\theta > \pi /2$).
Moreover, 
	\begin{align*}
		\LongSet{
			(M_w^0, M_\sigma^0) (\zeta, \tilde{a}, \nu)
		}{ 
			(\zeta, \tilde{a}, \nu) 
			\in \overline{D_\zeta} \times \overline{D_{\tilde{a}}} \times \DomainOfLambda
		}
	\end{align*}
is $\mathcal{R}$-bounded since $E$ and $F$ are Hilbert spaces.
Set $M^0 = (M_w^0, M_\sigma^0)$ and
	\begin{align*}
		L_1 
		& = - i \nabla^\prime (\eta  + (- \Delta^\prime)^{m})^{- 1/2m}, \\
		L_2 
		& = \eta^{1/2m} (\eta + (- \Delta^\prime)^{m})^{ - 1 / 2m}, \\
		L_3 
		& = \partial_t  (\eta + (- \Delta^\prime)^m))^{ - l / 2 m}.
	\end{align*}
Then, boundedness and analyticity of $M^0$ with respect to $\zeta$ and $a$ leads
	\begin{align*}
		 \mathcal{R} \left \{
			{\xi^{\prime}}^\alpha \partial_{\xi^\prime}^\alpha M^0 \left( 
				\frac{\xi^\prime}{( \eta + \abs{\xi^\prime}^{2m})^{1/2m}}, 
				\frac{1}{ (\eta + \abs{\xi^\prime}^{2m})^{1/2m}}, 
				\nu 
			\right)
		: 
			\xi^\prime \in \Real^{n-1} \setminus \{ 0 \}, \,
			\alpha \in \{0, 1 \}^{n-1}, \,
			\nu \in \DomainOfLambda
		\right \}
		< \infty.
	\end{align*}
Thus, the operator-valued Fourier multiplier theorem implies 
	\begin{align*}
		 M^0 (L_1, L_2, \nu) \in \mathcal{B} (Y_F \times Y^{m}_E ; Y^{2m}_E \times Y_F)
	\end{align*}
and $\LongSet{M^0 (L_1, L_2, \nu) }{\nu \in \DomainOfLambda}$ is $\mathcal{R}$-bounded on $\mathcal{B} (Y_F \times Y^{m}_E ; Y^{2m}_E \times Y_F)$.
Finally, because of analyticity of $M^0 (L_1 , L_2 , \cdot)$, we can use the Kalton--Weis theorem to find $M^0(L_1, L_2, L_3) \in \mathcal{B} (Y_F \times Y^{m}_E ; Y^{2m}_E \times Y_F)$.
\end{proof}
	
We find from Proposition \ref{prop_boundedness_of_canonical_extension}, Proposition \ref{prop_boundedness_multiplier_operator} and
	\begin{align*}
		&\frac{d}{dt} + (\eta +(- \Delta')^{m})^{l / 2m} \\ 
		&\quad \in {\rm Isom}(W^1_p(J; W^{2m-1/p}_{q} (\IR^{n-1}; F)) 
		\cap L_p (J; W^{l + 2m- 1/p}_{q}(\IR^{n-1}; F)); Y_F), \\
		& \left(
			\eta + (- \Delta')^{m} 
		\right)^{m_0 / 2 m} \\ 
		&\quad \in {\rm Isom}(W^1_p(J; W^{2m-1/p}_{q} (\IR^{n-1}; F)) 
		\cap L_p (J; W^{l + 2m- 1/p}_{q}(\IR^{n-1}; F)); Z_\rho), \\
		& \left(
			\eta + (- \Delta')^{m}
		\right)^{m_0  / 2 m} 
		\in {\rm Isom}(Y_F; Y_0) \\
		& \left(
			\eta + (- \Delta')^{m}
		\right)^{m_j / 2 m} 
		\in {\rm Isom}(Y_F; Y_j) \quad (j=1, \cdots, m), 
	\end{align*}
that
	\begin{align*}
		&\|u (0) \|_{Y_E} 
		+ \left\|
			\left(
			\frac{d}{dt} 
			+ (\eta +(- \Delta')^{m})^{l/2m}
			\right)
			\left(
			\eta + (- \Delta')^{m}
		\right)^{-m_0  / 2 m}  \rho
		\right\|_{Y_F} \\
		\le &~C \left(
			\|
			\left(
			\eta + (- \Delta')^{m}
		\right)^{-m_0  / 2 m}  g_0
			\|_{Y_F} + \sum_{j=1}^m \|
					\left(
			\eta + (- \Delta')^{m}
		\right)^{-m_j / 2 m}  g_j
			\|_{Y_E}
		\right), 
	\end{align*}
where $u (0) := u |_{y=0}$
This leads the desired maximal $L_p$ regularity 
	\begin{align*}
		\eta \Vert
			u
		\Vert_{X}
		+ \|u\|_{Z_u} 
		+ \| \rho \|_{Z_\rho} 
		\le ~C \left(
			\|g_0\|_{Y_0} + \sum_{j=1}^m \|g_j\|_{Y_j}
		\right)
	\end{align*}
for some $C>0$.
We finish proving Theorem \ref{thm_main} in the case of the half space.

\subsection{The case of a domain with a compact boundary}\label{Sec: General}

Let us consider the case of a bounded domain $G$ and an exterior.
The proof is based on (i) solving the case of variable coefficient with lower order terms, (ii) localization procedure and coordinate transform.
Since this method is well-known, we do not give a detail of the proof, see \cite{DenkHieberPruss2003,Denk-Pruss-Zacher2008,KunstmannWeis2004,PrussSimonett2016} for example.
We show only outline of the proof.
Note that conditions (E), (LS) and (ALS) are invariant under the coordinate transform.
First we give estimates for lower-order terms.
	\begin{proposition}\label{prop_lower}
Let $a_\alpha, b_{j\beta}, c_{j\gamma}$ satisfy (SA), (SB) and (SC), then there exists $C>0$ such that 
		\begin{align*}
			\|a_\alpha D^\alpha u\|_X 
			&\le C\|a_\alpha\|_{L_\infty(J; L_{r_\alpha}(G))} \|u\|_{Z_u} ,
			\quad (\abs{\alpha} \leq 2m -1) \\
			\|b_{j\beta} D^\beta u\|_{Y_j} 
			&\le C\|b_{j\beta}\|_{L_\infty(J; (L_{s_{j\beta}}\cap B^{2m\kappa_j}_{r_{j\beta}, q})(\g))} \|u\|_{Z_u} ,
			\quad (\abs{\beta} \leq m_j -1) \\
			\| c_{j\gamma} D^\gamma_\g \rho\|_{Y_j} 
			&\leq C \|
				c_{j\gamma}
			\|_{
				L_{t_{j\gamma}}(J; L_{s_{j\gamma}}(\g)) 
				\cap L_{\tau_{j\gamma}}(J; B^{2m\kappa_j}_{r_{j\gamma}, q}(\g))
			} \|
				\rho
			\|_{Z_\rho},
			\quad (\abs{\gamma} \leq k_j -1)
		\end{align*}
	\end{proposition}

	\begin{proof}
		First, for each $|\alpha|=k\le 2m-1$, the assumption (SA) derives 
		\begin{align*}
			\|
				a_\alpha D^\alpha u
			\|_{L_q(G)} 
			&\le \|
				a_\alpha
			\|_{ L_{r_\alpha}(G)} \|
				D^\alpha u
			\|_{L_{r'_\alpha}(G)}\\
			& \le \|
				a_\alpha
			\|_{L_{r_\alpha}(G)} \|
				u
			\|_{W^{2m}_q(G)}, 
		\end{align*}
where $1/q=1/r_\alpha+ 1/r'_\alpha$ and we use the embedding $ W^{2m}_q(G) \hookrightarrow W^k_{r'_\alpha}(G)$. 
This means 
		\begin{align*}
			\|
				a_\alpha D^\alpha u
			\|_{X} 
			&\le \|
				a_\alpha
			\|_{L_\infty(J; L_{r_\alpha}(G))}\|
				u
			\|_{Z_u}.
		\end{align*}
Second, for each $|\beta|=k\le m_j -1$, we find from paraproduct formula, definition of Besov spaces on a domain and the assumption (SB) that
		\begin{align}
			\|
				b_{j\beta} D^\beta u
			\|_{W^{2m\kappa_j}_{q}(\g)} \notag 
			&\le C(\|
				b_{j\beta}
			\|_{L_{s_{j\beta}}(\g)} \|
				D^\beta u 
			\|_{B^{2m\kappa_j}_{s'_{j\beta}, q}(\g)} 
			+ \|
				b_{j\beta}
			\|_{B^{2m\kappa_j}_{r_{j\beta}, q}(\g)} \|
				D^\beta u
			\|_{L_{r'_{j\beta}}(\g)}) \\ \label{eq_b_j_beta_D_beta_u} \notag \\
			 &\le C \|
				b_{j\beta}
			\|_{(L_{s_{j\beta}}\cap B^{2m\kappa_j}_{r_{j\beta}, q})(\g)} \|
				u 
			\|_{W^{2m-1/p}_{q}(\g)}, 
		\end{align}
where $1/q = 1/s_{j\beta} + 1/s'_{j\beta} = 1/r_{j\beta} + 1/r'_{j\beta}$ and we used the embeddings
		\begin{align*}
			{\rm tr}|_\g u 
			&\in W^{2m-1/q}_{q}(\g)
			\hookrightarrow (B^{k + 2m\kappa_j}_{s'_{j\beta}, q} 
			\cap W^k_{r'_{j\beta}})(\g). 
		\end{align*}
Moreover, we have $\mathrm{tr}|_{\gamma} u \in L_p (J ; W^{2m - 1/q}_q (\Gamma))$ for $u \in Z_u$.
This means 
		\begin{align*}
			\|b_{j\beta} D^\beta u\|_{Y_j} &\le C\|b_{j\beta}\|_{L_\infty(J; (L_{s_{j\beta}} \cap W^k_{r'_{j\beta}})(\g))} \|u\|_{Z_u}. 
		\end{align*}
At last, for each $|\gamma|=k \le k_j - 1$, it follows from the same way as for (\ref{eq_b_j_beta_D_beta_u}) that
		\begin{align*}
			 \|
				c_{j\gamma} D^\gamma_\g \rho
			\|_{W^{2m\kappa_j}_{q}(\g)} 
			 \le C \left(
				\|
					c_{j\gamma}
				\|_{L_{s_{j\gamma}^c}(\g)}\|
					D^\gamma_\g \rho 
				\|_{B^{2m\kappa_j}_{s^{c'}_{j\gamma}, q}(\g)} 
				+ \|
					c_{j\gamma}
				\|_{B^{2m\kappa_j}_{r_{j\gamma}^c, q}(\g)} \|
					D^\gamma_\g \rho
				\|_{L_{r_{j\gamma}^{c'}}(\g)}
			\right)
		\end{align*}
where $1/q = 1/{s_{j\gamma}^c} + 1/{s^{c'}_{j\gamma}} = 1/{r_{j\gamma}^c} + 1/{r^{c'}_{j\gamma}}$. 
Integral in time and use H\"older's inequality, 
		\begin{align*}
			\|
				c_{j\gamma} D^\gamma_\g \rho
			\|_{Y_j} 
			& \le C\left(
				\|
					c_{j\gamma}
				\|_{L_{t_{j\gamma}}(J; L_{s_{j\gamma}^c}(\g))}
				\|
					D^\gamma_\g \rho 
				\|_{L_{t'_{j\gamma}}(J; B^{2m\kappa_j}_{s^{c'}_{j\gamma}, q}(\g))} 
			+ \|
					c_{j\gamma}
				\|_{L_{\tau_{j\gamma}}(J; B^{2m\kappa_j}_{r_{j\gamma}^c, q}(\g))} \|
					D^\gamma_\g \rho
				\|_{L_{\tau'_{j\gamma}}(J; L_{r_{j\gamma}^{c'}}(\g)})\right)
		\end{align*}
where $1/p = 1/t_{j\gamma} + 1/t_{j\gamma} = 1/\tau_{j\gamma} + 1/\tau'_{j\gamma}$. 
Here we use the mixed derivative theorems 
		\begin{align*}
			Z_\rho 
			= W^1_p(J; W^{2m\kappa_0}_{q}(\g)) \cap L_p(J; W^{l + 2m\kappa_0}_{q} (\g)) 
			= \bigcap_{0\le s \le 1} W^s_p(J; W^{l(1-s) + 2m\kappa_0}_{q}(\g)). 
		\end{align*}
The assumption (SC) ensures the existence of $s \in [0, 1]$ such that 
		\begin{align*}
			W^s_p(J; W^{l(1-s) + 2m\kappa_0}_{q}(\g))
			\hookrightarrow L_{t'_{j\gamma}}(J; B^{k + 2m\kappa_j}_{s^{c'}_{j\gamma}, q}(\g)) ,L_{\tau'_{j\gamma}}(J; W^k_{r_{j\gamma}^{c'}}(\g)), 
		\end{align*}
respectively. 
This means 
		\begin{align*}
			\|
				c_{j\gamma} D^\gamma_\g \rho
			\|_{Y_j} 
			&\le C\|
				c_{j\gamma}
			\|_{L_{t_{j\gamma}}(J; L_{s_{j\gamma}^c}(\g)) \cap L_{\tau_{j\gamma}}(J; B^{2m\kappa_j}_{r_{j\gamma}^c, q}(\g))} \|\rho\|_{Z_\rho}. 
		\end{align*}
	\end{proof}
	
	\begin{proposition} \label{prop_qs_variable_coefficient}
Let $J=[0, T]$, $G = \IR^n_+$ and $\g=\partial G$, $1<p, q<\infty$ and $E$ and $F$ be separable Hilbert spaces. 
Let assumptions $(E)$,  $(SA)$,  $(SB)$,  $(SC)$,  $(LS)$ and $(ALS)$ hold.
Assume $\mathcal{A}$, $\mathcal{B}_j$ and $\mathcal{C}_j$ are given by
		\begin{align*}
			\mathcal{A} ( t, x, D)
			& = \mathcal{A}_\#(D) 
			+ \mathcal{A}^{small} (t,  x, D)
			+ \mathcal{A}^{low}( t, x, D) , \\
			\mathcal{B}_j (t, x, D)
			& = \mathcal{B}_{j \#  } (D)
			+ \mathcal{B}^{small}_j (t, x, D)
			+ \mathcal{B}^{low}_j (t, x, D), \\
			\mathcal{C}_j (t, x^\prime, D_\Gamma)
			& = \mathcal{C}_{j \#} (D_\Gamma) 
			+ \mathcal{C}^{small}_j (t, x^\prime, D_\Gamma)
			+ \mathcal{C}^{low}_j (t, x^\prime, D_\Gamma),
		\end{align*}
where the equation with $\mathcal{A}_\#$, $\mathcal{B}_{j \# }$ and $\mathcal{C}_{j \# }$ satisfy (LS) and (ALS), $\mathcal{A}^{low}$, $\mathcal{B}_{j }^{low}$ and $\mathcal{C}_{j }^{low}$ $(j = 0, \cdots , m)$ are lower order terms and 
		\begin{align*}
			\Vert
				\mathcal{A}^{small} ( x, D) u
			\Vert_{X} 
			& \leq \delta \Vert
				u
			\Vert_{Z_u} ,\\ 
			\Vert
				\mathcal{B}^{small}_j (t, x, D) u
			\Vert_{Y_j} 
			& \leq \delta \Vert
				u
			\Vert_{Z_u}
			\quad (j = 0, \cdots , m) ,\\
			\Vert
				\mathcal{C}^{small}_j (t, x, D) \rho
			\Vert_{Y_j} 
			& \leq \delta \Vert
				\rho
			\Vert_{Z_\rho}
			\quad (j = 0, \cdots , m),
		\end{align*}
for sufficiently small $\delta >0$.
Then, there exist positive constants $\eta_0 > 0$, $C$ and $C_T$, for 
		\begin{align*}
			(f, \{g_j\}_{j=0}^m, \rho_0) 
			\in X \times \prod_{j=0}^m Y_j \times \pi Z_\rho,
		\end{align*}
if $\eta \geq \eta_0$, the equations (\ref{eq_qs}) admits a unique solution $(u, \rho) \in Z_u \times Z_\rho$ such that
		\begin{align}
			\eta \Vert
				u
			\Vert_{X}
			+ \Vert
				u
			\Vert_{Z_u}
			+ \Vert
				\rho
			\Vert_{Z_\rho}
			\leq C \Vert
				f
			\Vert_{X}
			+ C \sum_{j= 0}^m \Vert
				g_j
			\Vert_{Y_j}
			+  C_T \Vert
				\rho_0
			\Vert_{\pi Z_\rho}. \label{eq_max_reg_half_space_variable_coef}
		\end{align}
	\end{proposition}
	\begin{proof}
Assume $|J|$ is small, where $|J|$ is the length of $J$.
Clearly, $\Vert u \Vert_{D (\mathcal{A}^{low})} \leq \Vert u \Vert_{D (\mathcal{A}_\#)}$, $\Vert u \Vert_{D (\mathcal{B}_j^{low})} \leq \Vert u \Vert_{D (\mathcal{B}_{j \#})}$ and $\Vert u \Vert_{D (\mathcal{C}_j^{low})} \leq \Vert u \Vert_{D (\mathcal{C}_{j\#})}$.
Thus, if we take $\eta > 0$ sufficiently large, we find from the space-time Sobolev embedding, which enable us to estimate lower-order terms as small perturbation since $|J|$ is small, and the Neumann series argument that can be estimated $\mathcal{A}^{small}$, $\mathcal{A}^{low}$, $\mathcal{B}^{small}_j$, $\mathcal{B}^{low}_j$, $\mathcal{C}^{small}_j$ and $\mathcal{C}^{low}_j$ as relatively small perturbations.
For $J$ with arbitrary finite length, we can divide $J$ into finite short intervals.
For these short intervals, we can apply the same argument as above step by step to get (\ref{eq_max_reg_half_space_variable_coef}).
	\end{proof}

Now we prove Theorem \ref{thm_main}.
For the sake of simplicity, we consider the case of bounded domains.
The case of exterior domains is treated by a similar way.
Temporarily, we assume $|J|$ is small.
Let $\delta > 0$ be small.
Let us introduce an open covering of $G$ such that
	\begin{align*}
		G 
		& \subset \cup_{k = 0}^{N} U_k \\
		U_k
		& = B (x_k, \delta), \quad x_k \in G 
		\quad (k = 0, \cdots , M) \\
		U_k
		& = B (x_k, \delta), \quad x_k \in \partial G
		\quad (k = M + 1 , \cdots , N)
	\end{align*}
for some $M$ and $N$.
We also introduce a partition of unity $\{ \varphi_j \}_{j = 0}^N$ satisfying 
	\begin{align*}
		\varphi_j \in C_0^\infty (\Real^n), \quad
		0 \leq \varphi_j \leq 1, \quad
		\mathrm{spt} \, \varphi_j \subset U_j, \quad
		\sum_{j = 0}^N \varphi_j \equiv 1 \, \, \mathrm{on} \, \, \overline{G}
	\end{align*}

Suppose $(u, \rho) \in Z_u \times Z_\rho$ be a solution to (\ref{eq_qs}) with $\rho_0=0$, which is without loss of generality.
For $k \geq M + 1$, we apply the canonical coordinate transform, which is denoted by $\Phi_k$, from $U_k$ to local neighbourhood of the half space so that $U_k \cap \Gamma$ is flat.
Since coefficients are continuous, if we take $\delta > 0$ be sufficiently small beforehand, we can extend coefficients to the half space and write these extended coefficients as $a_\alpha^k$, $b_{j \beta}^k$ and $c_{j \gamma}^k$, to the half space so that 
	\begin{align*}
		\Vert a^k_\alpha - a_\alpha (0, x_k) \Vert_{L_\infty (J \times \Real^n_+ ; \cB{(E)})}, \quad
		\Vert b^k_{j \beta} - b_{j \beta} (0, x_k) \Vert_{L_\infty (J \times \Real^{n-1} ; \cE_j)}, \quad \Vert c_{j \gamma}^k - c_{j \gamma} (0, x_k)\Vert_{L_\infty (J \times \Real^{n-1} ; \cF_j)}
	\end{align*}
are sufficiently small.
Put $(u_k, \rho_k, f_k, g_{j,k}) =  ( \phi_k u, \phi_k \rho, \phi_k f, \phi_k g_j)$ for $k = 0, \cdots , N$.
	Then, for $k = M+1, \cdots , N$, $(u_k, \rho_k)$ satisfies
	\begin{align} \label{eq_cut_off_qs}
		\left\{
			\begin{array}{rclll}
				\eta u_k
				+ \cA_{\#} u_k
				&=
				& F_k (f_k, u) \quad 
				&(t \in J, \, x \in G),
				& \, \\
				\partial_t \rho_k  
				+\cB_{0 \#} u_k + \cC_{0 \#} \rho_k
				&=
				& G_{0, k} (g_{0, k}, u, \rho) \quad 
				&(t \in J, \, x \in \Gamma), 
				& \, \\
				\cB_{j \#} u_k + \cC_{j \#} \rho_k 
				&=
				& G_{j, k} (g_{j, k}, u, \rho) \quad 
				& (t \in J, \, x \in \Gamma)
				& (j = 1, \cdots , m), \\
				\rho_k(0)
				&=
				& 0 \quad 
				&(x \in \Gamma), 
				& \,
			\end{array}
		\right.
	\end{align}
for
	\begin{align*}
		F_k (f_k, u)
		& = f_k 
		- \varphi_k \left(
			\mathcal{A} - \mathcal{A}_\# 
		\right) u
		+ \left [
			\mathcal{A}, \varphi_k
		\right ] u \\
		G_{0, k} (g_{0, k}, u, \rho)
		& = g_{0,k}
		- \varphi_k \left(
			\mathcal{B}_{0 \#} - \mathcal{B}_0
		\right) u
		+ \left[
			\mathcal{B}_0, \varphi_k
		\right] u  - \varphi_k \left(
			\mathcal{C}_{0 \#} - \mathcal{C}_0
		\right) \rho
		+ \left[
			\mathcal{C}_0, \varphi_k
		\right] \rho	\\
		G_{j, k} (g_{j, k}, u, \rho)
		& = g_{j,k}
		- \varphi_k \left(
			\mathcal{B}_{j \#} - \mathcal{B}_j
		\right) u
		+ \left[
			\mathcal{B}_j, \varphi_k
		\right] u  - \varphi_k \left(
			\mathcal{C}_{j \#} - \mathcal{C}_j
		\right) \rho
		+ \left[
			\mathcal{C}_j, \varphi_k
		\right] \rho
		\quad (j = 1, \cdots , m),
	\end{align*}
where $[ \cdot , \cdot ]$ is the commuter.
For $k = 0, \cdots , M$, we can solve
	\begin{align*}
		\eta u_k
		+ \cA_{\#} u_k
		= F_k (f_k, u) \quad 
		(t \in J, \, x \in \Real^n)
	\end{align*}
and find from Proposition {\ref{prop_lower}} and the regularity estimate of elliptic operators that
	\begin{align} \label{eq_estimate_u_k_cut_off_qs_1}
		\eta \Vert
			u_k
		\Vert_{X}
		+ \Vert
			u_k
		\Vert_{Z_u}
		\leq C_1 \Vert
			f
		\Vert_{X}
		+ \varepsilon \Vert
			u
		\Vert_{Z_u}
		+ C_\varepsilon \Vert
			u
		\Vert_{X}.
	\end{align}
for sufficiently small $\varepsilon > 0$.
Put $\rho_k = 0$ for $k = 0, \cdots , M$.
In the case of $k = M + 1, \cdots , N$,	we apply coordinate transform by $\Phi_k$ to (\ref{eq_cut_off_qs}).
The transformed problem is solvable by Proposition \ref{prop_qs_variable_coefficient}.
Pulling buck the solution to we obtain the solution to (\ref{eq_cut_off_qs}) such that
	\begin{align} \label{eq_estimate_u_k_cut_off_qs_2}
		\eta \Vert
			u_k
		\Vert_{X}
		+ \Vert
			u_k
		\Vert_{Z_u}
		+ \Vert
			\rho_k
		\Vert_{Z_\rho} 
		\leq  C \Vert
			f
		\Vert_{X}
		+ C \sum_{j = 0}^m \Vert
			g_j
		\Vert_{Y_j}
		+ \varepsilon \Vert
			u
		\Vert_{Z_u}
		+ C_\varepsilon \Vert
			u
		\Vert_{X}
		+ C \abs{J}^\alpha \Vert
			\rho
		\Vert_{Z_\rho},
	\end{align}
for some $\alpha > 0$.
	$\abs{J}^\alpha$ appears because $F_k (f_k, u) - f_k$ has only lower-order differential terms.
We denote by $\mathcal{S}^k : X \times Y_0 \times \cdots \times Y_m \rightarrow Z_u \times Z_\rho $ the solution operator of (\ref{eq_cut_off_qs}) with $\rho_0 = 0$, i.e. $(u_k, \rho_k) = \mathcal{S}^k (F_k , G_{0,k}, \cdots , G_{m, k})$.
On the other hand, it follow that
	\begin{align}
		(u, \rho) 
		& = \sum_{k = 0}^N (u_k, \rho_k)
		=  \sum_{k = 0}^N \mathcal{S}^k \left(
			F_k (f_k, u) , G_{0,k} (g_{0, k}, u, \rho), \cdots , G_{m, k} (g_{m, k}, u, \rho)
		\right) \notag \\
		& =  \sum_{k = 0}^N \mathcal{S}^k \left(
			F_k (f_k, 0) , G_{0,k} (g_{0, k}, 0, 0), \cdots , G_{m, k} (g_{m, k}, 0, 0)
		\right) \notag \\
		& \qquad +  \sum_{k = 0}^N \mathcal{S}^k \left(
			F_k (0 , u), G_{0,k} (0, u, \rho), \cdots , G_{m, k} (0, u, \rho)
		\right) \notag \\ \label{eq_sum_of_u_k_rho_k}
		& = : \mathcal{T} (f, g_0, \cdots, g_m)
		- R (u, \rho) ,
	\end{align}
where we write restriction of $(u_k, \rho_k)$ on $U_k$ as $(u_k, \rho_k)$ for simplicity.
	(\ref{eq_estimate_u_k_cut_off_qs_1}) and (\ref{eq_estimate_u_k_cut_off_qs_2}) imply
	\begin{align*}
		\Vert
			R (u, \rho)
		\Vert_{Z_u \times Z_\rho}
		\leq \frac{1}{2}
			\Vert
				u
			\Vert_{Z_u}
			+ C  \abs{J}^\gamma \Vert
				\rho
			\Vert_{Z_\rho}.
	\end{align*}
for some $\gamma > 0$.
Let $\mathcal{S}_0 : X \times Y_0 \times \cdots \times Y_m \rightarrow Z_u \times Z_\rho $ be the solution operator of (\ref{eq_qs}) with $\rho_0 = 0$, i.e. $(u, \rho) = \mathcal{S}_0 (f, g_0 , \cdots , g_m)$.
It follows from (\ref{eq_sum_of_u_k_rho_k})
	\begin{align}
		\mathcal{S}_0(f, g_0, \cdots, g_m)
		= \mathcal{T} (f, g_0, \cdots, g_m)
		- R \left(
			\mathcal{S}_0 (f, g_0, \cdots, g_m) 
		\right).
	\end{align}
Then, if we take $\abs{J}$ small and $\eta > 0$ large, we can use the Neumann series argument to get $\mathcal{S}_0 = (Id + R)^{-1} \mathcal{T}$ and
	\begin{align*}
		\Vert
			\mathcal{S}_0
		\Vert_{\mathcal{B}(X \times Y_0 \times \cdots \times Y_m  ; Z_u \times Z_\rho )}
		\leq C.
	\end{align*}
For $J$ with arbitrary finite length, we can divide $J$ into finite short interval.
For this short intervals, the same argument as above also works.

\section{Examples}\label{Examples}

In this section we give some examples for our problems. 
We especially focus on checking the Lopatinskii--Shapiro and asymptotic Lopatinskii--Shapiro conditions. 
Throughout this section, we assume $E=F=\IC$ and write the outer unit normal on the boundary by $\nu$.

{\bf Example \ref{Examples}.1} 
		\begin{align} \label{eq_example_1}
			\left\{
			\begin{array}{rll}
				\eta u - \Delta u 
				&= f \quad 
				&(t\in J,~ x\in G), \\
				\partial_\nu u + \partial_t \rho 
				&= g_0 \quad 
				&(t \in J,~ x\in \g), \\
				u - \rho
				& = 0
				&(t \in J,~ x\in \g) \\
				\rho(0,x) &= \rho_0(x) \quad &(x\in \g). 
			\end{array}
		\right.
		\end{align}
The equation of the Lopatinskii--Shapiro condition is 
		\begin{align*}
			\left\{
			\begin{array}{rll}
				(\eta +  |\xi'|^2 - \partial_y^2) v(y) &= 0 \quad &(y>0), \\
				-\partial_y v(0) + \lambda \sigma &=  h_0, \quad & \\
				v(0) - \sigma &= h_1. \quad &
			\end{array}
		\right.
		\end{align*}
The solution of the first equation $C_0(\IR_+; E)$ is given by $v(y) = e^{-\sqrt{\eta + |\xi'|^2}} v(0) = e^{- \mu y} v(0)$ for $\mu = (\eta + \abs{\xi^\prime}^2)^{1/2}$. 
The boundary conditions lead to the equation 
	\begin{align*}
		\begin{pmatrix}
		 \sqrt{\eta+|\xi'|^2} & \lambda \\
		 1 & -1
		\end{pmatrix}
		\begin{pmatrix}
		 v(0)\\
		 \sigma
		\end{pmatrix}
		= 
		\begin{pmatrix}
		 h_0\\
		 h_1
		\end{pmatrix}. 
		\end{align*}
We see that  the determinant of the matrix is $-\lambda - \sqrt{\eta+|\xi'|^2}\neq 0$ for $\eta>0$, $(\lambda, \xi') \in (\DomainOfLambda \times \IR^{n-1}) \setminus\{(0,0)\}$ ($\theta >\pi / 2$). 
Therefore, the Lopatinskii--Shapiro condition is satisfied. 

The equation of the first asymptotic Lopatinskii--Shapiro condition is
	\begin{align*}
		\left \{
			\begin{array}{rcll}
				( \eta + \abs{\xi^\prime}^2 - \partial_y^2 ) v (y)
				& = 
				& 0
				& (y > 0), \\
				v (0)
				&	= 
				& h_1.
				& \,
			\end{array}
		\right.
	\end{align*}
for $\eta > 0$ and $\xi^\prime \in \Real^{n-1}$.
The solution to this ODE is uniquely determined by $v (y) = e^{ - \mu y} h_1$.
	
The equation of the second asymptotic Lopatinskii--Shapiro condition is 
	\begin{align*}
		\left \{
			\begin{array}{rcll}
				(\abs{\xi^\prime}^2 - \partial_y^2 ) v (y)
				& = 
				& 0
				& (y > 0), \\
				- \partial_y v (0) + \lambda \sigma
				& = 
				& h_0,
				& \, \\
				v (0) - \sigma
				&	= 
				& h_1.
				& \,
			\end{array}
		\right.
	\end{align*}
for $(\lambda, \xi^\prime) \in (\DomainOfLambda\cup\{0\})  \times \mathbb{S}^{n-2}$($\theta>\pi/2$).
The equation of the first equation implies $v (y) = e^{ - \abs{\xi^\prime} y} v(0)$, and thus $- \partial_y v (0) = \abs{\xi^\prime} v (0)$.
Since the determinant of the matrix
	\begin{align*}
		\left(
			\begin{array}{cc}
				\abs{\xi^\prime}
				& \lambda \\
				1
				& -1
			\end{array}
		\right)
	\end{align*}
is never zero by the choice of $(\lambda, \xi^\prime)$.

The equation of the third asymptotic Lopatinskii--Shapiro condition is
		\begin{align*}
		\left \{
			\begin{array}{rcll}
				( \abs{\xi^\prime}^2 - \partial_y^2 ) v (y)
				& = 
				& 0
				& (y > 0), \\
				v (0)
				&	= 
				& h_1
				& \,
			\end{array}
		\right.
	\end{align*}
for $\xi^\prime \in \mathbb{S}^{n-2}$.
This equation is uniquely determined by $v (y) = e^{- \abs{\xi^\prime} y} h_1$.
Thus, the Lopatinskii--Shapiro and asymptotic Lopatinskii--Shapiro conditions are satisfied, and (\ref{eq_example_1}) is solvable in the maximal regularity space.

{\bf Example \ref{Examples}.2} 
		\begin{align}
			\left\{
			\begin{array}{rll}
				\eta u - \Delta u 
				&= f \quad 
				&(t\in J,~ x\in G), \\
				\partial_\nu u + \partial_t \rho - \Delta_\g \rho 
				&= g_0 \quad 
				&(t \in J,~ x\in \g), \\
				u - \rho 
				&= g_0 \quad 
				&(t \in J,~ x\in \g), \\
				\rho(0,x) 
				&= \rho_0(x) \quad 
				&(x\in \g). 
			\end{array}
		\right.
		\end{align}
The equation of the Lopatinskii--Shapiro condition is
	\begin{align*}
		\left \{
			\begin{array}{rll}
				( \eta  + \abs{\xi^\prime}^2 - \partial_y^2 ) v (y)
				& = 0
				& (y > 0) \\
				- \partial_y v (0) + \lambda \sigma + \abs{\xi^\prime}^2 v
				& = h_0
				& \, \\
				v (0) - \sigma
				& = h_1.
				& \,			
			\end{array} 
		\right.
	\end{align*}
We find $v (y) = e^{- \mu y} v (0) $ for $\mu = (\eta + \abs{\xi^\prime}^2)^{1/2}$ and
	\begin{align*}
		\mathrm{det} \left(
			\begin{array}{cc}
				\mu
				& \lambda + \abs{\xi^\prime}^2 \\
				1
				& -1
			\end{array}
		\right)
		\neq 0
	\end{align*}
for $(\lambda, \xi') \in (\DomainOfLambda \times \IR^{n-1}) \setminus\{(0,0)\}$ ($\theta >\pi / 2$).
Thus Lopatinskii--Shapiro condition is satisfied.

Let us check asymptotic Lopatinskii--Shapiro conditions.
The equation of the first and third asymptotic Lopatinskii--Shapiro conditions are
	\begin{align*}
		\left \{
			\begin{array}{rll}
				( \eta  + \abs{\xi^\prime}^2 - \partial_y^2 ) v (y)
				& = 0
				& (y > 0) \\
				v (0) - \sigma
				& = h_1
				& \,			
			\end{array} 
		\right.
	\end{align*}
for $\xi^\prime \in \Real^{n-1}$ and
	\begin{align*}
		\left \{
			\begin{array}{rll}
				( \abs{\xi^\prime}^2 - \partial_y^2 ) v (y)
				& = 0
				& (y > 0) \\
				v (0) - \sigma
				& = h_1
				& \,			
			\end{array} 
		\right.
	\end{align*}
for $\xi^\prime \in \mathbb{S}^{n-2}$.
By the same way as above, we find this equation is uniquely solvable.

The equation of the second asymptotic Lopatinskii--Shapiro condition is
	\begin{align*}
		\left \{
			\begin{array}{rll}
				( \abs{\xi^\prime}^2 - \partial_y^2 ) v (y)
				& = 0
				& (y > 0) \\
				- \partial_y v (0) + \lambda \sigma + \abs{\xi^\prime}^2 v
				& = h_0
				& \, \\
				v (0)
				& = h_1.
				& \,			
			\end{array} 
		\right.
	\end{align*}
$v (y)$ is determined by the first and third equations, and $\sigma$ is uniquely determined by the second equation for $(\lambda, \xi^\prime)$.

{\bf Example \ref{Examples}.3}
The third example is the Cahn--Hilliard equations with the dynamic boundary condition and surface diffusion
		\begin{align}
			\left\{
			\begin{array}{rll}
				\eta u + \Delta^2 u 
				&= f \quad 
				&(t\in J,~ x\in G), \\ 
				\partial_t u + \partial_\nu \rho - \Delta_\g \rho 
				& = g_0 
				&(t\in J,~ x\in \g), \\
				\partial_\nu \Delta u 
				& = g_1 
				&(t\in J,~ x\in \g), \\
				u - \rho 
				& = g_2 
				&(t\in J,~ x\in \g), \\
				\rho(0, x) 
				&= \rho_0(x) 
				&(x\in \g). \\
			\end{array}
		\right.
		\end{align}
The equation of the Lopatinskii--Shapiro condition is 
		\begin{align*}
			\left\{
			\begin{array}{rll}
				(\eta +  (|\xi'|^2 - \partial_y^2)^2) v(y) &= 0 \quad &(y>0), \\
				-\partial_y v(0) + (\lambda + |\xi'|^2) \sigma &=  h_0, \quad & \\
				-\partial_y (|\xi'|^2 - \partial_y^2) v(0) &= h_1, \quad &\\
				v(0) - \sigma & = h_2. \quad &
			\end{array}
		\right.
		\end{align*}
The solution of the first equation which belongs to $C_0(\IR_+; E)$ is $v(y) = C_1 e^{- z_1 y} + C_2 e^{- z_2 y}$ with $z_{1,2}:= \sqrt{|\xi|^2\pm \eta^{1/2} i}$ and $C_{1,2}\in \IC$. 
Note that the real parts of $z_1$ and $z_2$ are non-negative. 
The boundary conditions lead
		\begin{align*}
		\begin{pmatrix}
		 z_1 & z_2 & \lambda + |\xi'|^2\\
		 z_1(z_1^2 - |\xi'|^2) & z_2(z_2^2 - |\xi'|^2) & 0\\
		 1 & 1 & -1
		\end{pmatrix}
		\begin{pmatrix}
		 C_1\\
		 C_2\\
		 \sigma
		\end{pmatrix}
		= 
		\begin{pmatrix}
		 h_0\\
		 h_1\\
		 h_2
		\end{pmatrix}. 
		\end{align*}
We see that  the determinant of the matrix is $i \eta^{1/2} \left( (\lambda + \abs{\Prime{\xi}}) (z_1 + z_2 ) + 2 z_1 z_2 \right)$,
and this is not zero for $\eta>0$, $(\lambda, \xi') \in (\DomainOfLambda \times \IR^{n-1})\setminus\{(0,0)\}$ ($\theta > \pi / 2$). 
Therefore Lopatinskii--Shapiro condition is satisfied. 

Let us check asymptotic Lopatinskii--Shapiro conditions.
	\begin{align*}
		\left \{
			\begin{array}{rll}
				\eta v (y) + (\abs{\xi^\prime}^2 - \partial_y^2) v (y)
				& = 0
				& (y > 0) \\
				\partial_y (\abs{\xi^\prime}^2 - \partial_y^2) v (0)
				& = h_1
				& \, \\
				v (0)
				& = h_2. 
				& \, 
			\end{array}
		\right.
	\end{align*}
The solution is of the form $v (y) = C_1 e^{-z_1 y} + C_2 e^{-z_2 y}$.
Since
	\begin{align*}
		\mathrm{det} \left(
			\begin{array}{cc}
				- \abs{\xi^\prime}^2 - z_1^3
				& - \abs{\xi^\prime}^2 - z_2^3 \\
				1 
				& 1
			\end{array}
		\right)
		= - i \eta^{1/2} \left(
			\frac{2 \abs{\xi^\prime}^2}{z_1 + z_2} + z_1 + z_2
		\right)
		\neq 0
	\end{align*}
for $\xi^\prime \in \Real^{n-1}$, the first asymptotic Lopatinskii--Shapiro condition is satisfied.

The equation of the second asymptotic Lopatinskii--Shapiro condition is
	\begin{align*}
		\left \{
			\begin{array}{rll}
				(\abs{\xi^\prime}^2 - \partial_y^2) v (y)
				& = 0
				& (y > 0) \\
				- \partial_y v (0) + \lambda \sigma + \abs{\xi^\prime}^2 \sigma
				& = h_0
				& \, \\
				\partial_y (\abs{\xi^\prime}^2 - \partial_y^2) v (0)
				& = h_1
				& \, \\
				v (0)
				& = h_2. 
				& \, 
			\end{array}
		\right.
	\end{align*}
The solution is of the form $v (y) = C_1 e^{- \abs{\xi^\prime} y} + C_2 y e^{- \abs{\xi^\prime} y}$.
Since 
	\begin{align*}
		\mathrm{det} \left(
			\begin{array}{ccc}
				\abs{\xi^\prime}
				& -1
				& \lambda + \abs{\xi^\prime}^2 \\
				0
				& 2 \abs{\xi^\prime}^2
				& 0 \\
				1
				& 0 
				& 0				
			\end{array}
		\right)
		= - 2 \abs{\xi^\prime} (\lambda + \abs{\xi^\prime}^2)\neq0
	\end{align*}
for $(\lambda, \xi') \in (\DomainOfLambda\cup\{0\}) \times \IS^{n-2}$ ($\theta > \pi / 2$), it holds.

The equation of the third asymptotic Lopatinskii--Shapiro condition is 	
		\begin{align*}
		\left \{
			\begin{array}{rll}
				(\abs{\xi^\prime}^2 - \partial_y^2) v (y)
				& = 0
				& (y > 0) \\
				\partial_y (\abs{\xi^\prime}^2 - \partial_y^2) v (0)
				& = h_1
				& \, \\
				v (0)
				& = h_2. 
				& \, 
			\end{array}
		\right.
	\end{align*}
We find from the same way as above this equation admits a unique solution for $\xi^\prime\in \mathbb{S}^{n-2}$.\\

{\bf{Acknowledgements}}
The first author was supported by the Program for Leading Graduate Schools, MEXT, JAPAN. 
Second author was supported by JSPS KAKENHI Grant Number 19K23408.


\end{document}